\documentclass{amsart}
\newcommand*{\thisisamsart}{}
\usepackage{amsfonts}
\usepackage{amsmath}
\usepackage{amsthm}
\usepackage{amssymb}
\usepackage{amscd}

\usepackage{amsxtra}
\usepackage{enumitem}

\usepackage{hyperref}
\theoremstyle{plain}

\newtheorem{thm}{Theorem}[subsection]
\numberwithin{thm}{section}
\newtheorem{lem}[thm]{Lemma}
\newtheorem{prop}[thm]{Proposition}

\newtheorem{cor}[thm]{Corollary}
\theoremstyle{definition}
\newtheorem{defn}[thm]{Definition}

\theoremstyle{remark}
\newtheorem{rmk}[thm]{Remark}
\newtheorem{eg}[thm]{Example}

\newcommand{\Spec}{\mathrm{Spec}\,}
\newcommand{\Gal}{\mathrm{Gal}\,}

\ifdefined\thisisamsart
\newcommand{\ead}{\email}
\fi 
\begin{document}
\title{An equivalence between two approaches to limits of local fields}
\author{Jeffrey Tolliver}
\ead{tolliver@ihes.fr}
\address{Department of Mathematics, Johns Hopkins University\\ 3400 N Charles St,Baltimore, MD, USA 21218
\ifdefined\thisisamsart \newline \else \\ \fi 
Present Address: Institut des Hautes \'Etudes Scientifiques, 35 Route de Chartres\\ Bures-sur-Yvette France 91440}

\begin{abstract}Marc Krasner proposed a theory of limits of local fields in which one relates the extensions of a local field to the extensions of a sequence of related local fields.  The key ingredient in his approach was the notion of valued hyperfields, which occur as quotients of local fields.  Pierre Deligne developed a different approach to the theory of limits of local fields which replaced the use of hyperfields by the use of what he termed triples, which consist of truncated discrete valuation rings plus some extra data.  We study the relationship between Krasner's valued hyperfields and Deligne's triples.\end{abstract}

\maketitle
\section*{Acknowledgements}
 This project was suggested by Caterina Consani, and was funded by the Johns Hopkins University.
\section{Introduction}
In this paper, the term local field will denote a field which is complete with respect to a discrete valuation and has a perfect residue field of finite characteristic.  The following classical theorem is the fundamental result of local class field theory.

\begin{thm}\label{lcft}Let $K$ be a local field which has a finite residue field.  Then $\Gal(K^{s}/K)^{ab}$ is the profinite completion of $K^\times$.  Furthermore, $(\Gal(K^s/K)/\Gal(K^s/K)^i)^{ab}$ is the profinite completion of $K^\times/(1+\mathfrak{m}_K^i)$\footnote{$\Gal(K^s/K)^i\subseteq \Gal(K^s/K)$ denotes the $i$-th upper ramification subgroup of the absolute Galois group. The reader is referred to chapter IV of \cite{serre} for a definition.}.\end{thm}

An incredible feature of the first part of the above theorem is that it gives a description of the abelian extensions in terms of the multiplicative structure on $K$, without making use of the addition operation.  It seems that it would be too much to hope for a nice description of the nonabelian extensions without using both the addition and multiplication operations.

The second part of the theorem above gives a description of abelian extensions $L/K$ satisfying the ramification condition $\Gal(L/K)^i=1$ (or the equivalent condition that $\Gal(K^s/L)\supseteq Gal(K^s/K)^i$).  The classification of such extensions depends solely on the multiplicative structure of the quotient $K/(1+\mathfrak{m}_K^i)=\{0\}\cup K^\times/(1+\mathfrak{m}_K^i)$ of $K$.  By analogy with the above paragraph, one might hope to understand all separable extensions which satisfy this ramification condition in terms of both the addition and multiplication operations on $K/(1+\mathfrak{m}_K^i)$.  One immediately runs in to the problem that the addition operation on this quotient is not well-defined.   Equivalently it may be regarded as a multivalued operation.  Because of this, we must introduce the definition of a \emph{hyperfield}, which is an analogue of a field with a multivalued addition operation.  Hyperfields were first defined by M. Krasner, and were inspired by the definition of a hypergroup by F. Marty\cite{krasner}\cite{marty}.

\begin{defn}A canonical abelian hypergroup $H$ consists of a set $H$ together with a multivalued operation $+:H\times H\rightarrow 2^H$ satisfying the following axioms\footnote{$2^H$ denotes the power set of $H$.}:

\begin{enumerate}[label=(\alph*)]
\item $(x+y)+z=x+(y+z)$  for all $x,y,z\in H$, where the left side is defined to mean $\bigcup_{t\in x+y} t+z$, and the right side is defined similarly.
\item $x+y=y+x$ for all $x,y\in H$.
\item There exists $0\in H$ such that $x+0=\{x\}$ for all $x\in H$.
\item For all $x\in H$ there is a unique element $-x\in H$ such that $0\in x+(-x)$.
\item For all $x,y,z\in H$, the inclusion $x\in y-z$ holds if and only if $y\in x+z$ holds.
\end{enumerate}

A multiring consists of a canonical abelian hypergroup $H$ together with a commutative associative unital operation $\cdot:H\times H\rightarrow H$ satisfying $(x+y)z\subseteq xz+yz$\footnote{The reader should note that unlike addition, the multiplication is single valued.}.  It is a hyperfield if every non-zero element has a multiplicative inverse and if $0\neq 1$.\end{defn}

The quotient $K/(1+\mathfrak{m}_K^i)$ carries the structure of a hyperfield in a canonical way, and this example is what motivated M. Krasner to define hyperfields in \cite{krasner}.  Hyperfields have more recently appeared in diverse settings in the work of A. Connes, C. Consani, M. Marshall, and O. Viro.  Because the reader is not expected to be familiar with such objects we shall give several examples in section \ref{examplessect} below.

The discussion following Theorem \ref{lcft} leads one to conjecture the following theorem, which follows immediately by combining the main result of this paper with the result of P. Deligne in \cite{deligne}.  It is worth noting that it is possible to have $K/(1+\mathfrak{m}_K^i)\cong F/(1+\mathfrak{m}_F^i)$ when $K$ and $F$ are distinct local fields.  Thus the following theorem (or the aforementioned result of P. Deligne) can provide a link between extensions of one local field with those of another, even when the fields involved have different characteristics.

\begin{thm}\label{delhypthm}Let $K$ be a local field.  Then the category of finite separable extensions $L/K$ satisfying $\Gal(K^s/L)\supseteq Gal(K^s/K)^i$ depends only on the isomorphism class of the hyperfield $K/(1+\mathfrak{m}_K^i)$.\end{thm}

Before discussing another motivation for studying the quotient $K/(1+\mathfrak{m}_K^i)$, we will need a few more definitions.   $K/(1+\mathfrak{m}_K^i)$ is not only a hyperfield but is also equipped with a non-archimedean absolute value, and this absolute value is related to the addition operation in a strong way.  M. Krasner introduced the following notion of valued hyperfield which abstracts this example, as well as a definition of homomorphisms of valued hyperfields.

\begin{defn}[{\cite[pg 144]{krasner}}]\label{defofvh}A valued hyperfield is a hyperfield equipped with a map $|\cdot|:H\rightarrow\mathbb{R}$ satisfying the following axioms:

(i) $|x|\geq 0$ with equality if and only if $x=0$.

(ii) $|xy|=|x| |y|$ for all $x,y\in H$.

(iii) $|x+y|\leq \mathrm{max}(|x|,|y|)$.

(iv) $|x+y|$ consists of a single element unless $0\in x+y$.  This axiom in particular implies that there is a well defined metric on $H$ given by $d(x,y)=|x-y|$ for $x\neq y$ and $d(x,x)=0$ for any  $x\in H$ (It is guaranteed to be a metric by axioms (i) and (iii)).

(v) There is a real number $\rho_H>0$ such that either $x+y$ is a closed ball of radius $\rho_H \mathrm{max}(|x|,|y|)$ for all $x,y\in H$, or $x+y$ is an open ball of radius $\rho_H \mathrm{max}(|x|,|y|)$ for all $x,y\in H$.  The smallest such $\rho_H$ is called the norm of the valued hyperfield.

Just as in the classical case, the valued hyperfield is said to be discretely valued if $0$ is the only non-isolated point in the image of the absolute value.\end{defn}

\begin{defn}[{\cite[pg 148]{krasner}}]\label{vhhom}A map $f:H_1\rightarrow H_2$ between valued hyperfields is called a homomorphism if the following axioms hold:

(i)  $f(xy)=f(x)f(y)$ for all $x,y\in H_1$.

(ii)  $f^{-1}(a+b)=f^{-1}(a)+f^{-1}(b)$ for all  $a,b\in f(H_1)$.\footnote{The reader familiar with hyperfields will note that this axiom is slightly stronger than the usual definition of a hypergroup homomorphism, which states that $f(x+y)\subseteq f(x)+f(y)$ for all $x,y\in H_1$.}

(iii)  $|f(x)|=|x|$ for all $x\in H_1$.

(iv)  The fiber over 1 is a ball.  Consequently, all fibers are balls.\end{defn}

M. Krasner was originally motivated to study the hyperfield $K/(1+\mathfrak{m}_K^i)$ in order to be able to define limits of local fields.  His idea was that one could study extensions of one local field $K$ by instead studying extensions of a whole sequence of local fields $K_i$, which do not necessarily have the same characteristic as $K$.

\begin{defn}Let $K$ be a local field and let $K_i$ be a local field for each $i\in\mathbb{N}$.  $K$ is said to be a limit of the sequence $\{K_i\}$ if there is an increasing sequence of natural numbers $\gamma_i$ such that there are isomorphisms $K/(1+\mathfrak{m}_K^{\gamma_i})\cong K_i/(1+\mathfrak{m}_{K_i}^{\gamma_i})$ of valued hyperfields.\end{defn}

Given a local field $K$ which is the limit of a sequence $\{K_i\}$ and given a finite separable extension $L/K$, M. Krasner has constructed the associated extension $L_i/K_i$. He has proven the following theorem.

\begin{thm}[{\cite[pg 201]{krasner}}]\label{krasthm}Let $K$ be the limit of a sequence of local fields $\{K_i\}$.  Let $L/K$ be a finite extension.  Let $L_i/K_i$ be the extensions of $K_i$ induced by $L/K$.  Then $L/K$ is Galois if and only if $L_i/K_i$ is Galois for all sufficiently large $i$.  In this case $\mathrm{Gal}(L_i/K_i)\cong \mathrm{Gal}(L/K)$ when $i$ is sufficiently large.\end{thm}

As an example of this phenomenon, it is shown in the author's PhD thesis that given a suitable infinite extension $L/K$, one may regard Wintenberger's field of norms $X_K(L)$ (as defined in \cite{wintenberger}) as a limit of the finite subextensions of $L/K$\cite{thesis}.  Furthermore, this thesis shows that one may recover Wintenberger's result which states that $\Gal(X_K(L)^s/X_K(L))\cong \Gal(L^s/L)$ from the theory of limits of local fields

P. Deligne has developed similar results to Theorems \ref{krasthm} and \ref{delhypthm} which replace the use of the hyperfield $K/(1+\mathfrak{m}_K^i)$ by the triple of data $(R,M,\epsilon)$ with $R=\mathcal{O}_K/\mathfrak{m}_K^i$, with $M=\mathfrak{m}_K/\mathfrak{m}_K^{i+1}$ and with $\epsilon:M\rightarrow R$ the map induced by the inclusion $\mathfrak{m}_K\subseteq\mathcal{O}_K$\cite{deligne}.  Just as M. Krasner defined a category of valued hyperfields which abstracts the properties of $K/(1+\mathfrak{m}_K^i)$, P. Deligne has defined a category of triples which abstracts the properties of $(\mathcal{O}_K/\mathfrak{m}_K^i,\mathfrak{m}_K/\mathfrak{m}_K^{i+1},\epsilon)$.  The key ingredient in his definition of a triple is the notion of a truncated DVR.  These notions are defined below.  For future convenience, we also define the valuation on a triple.

\begin{defn}[{\cite[1.1]{deligne}}]A truncated DVR\footnote{The name truncated DVR comes from the fact that such objects can be obtained as quotients of discrete valuation rings.} is a local Artinian ring whose maximal ideal is principal.  If $R$ is a truncated DVR and if $x\in R$, then we define $v_R(x)=\mathrm{sup} \{i\in \mathbb{N}\mid x\in \mathfrak{m}_R^i\}$, where $\mathfrak{m}_R$ is the maximal ideal.\end{defn}

\begin{defn}[{\cite[pg 126]{deligne}}]A triple $(R,M,\epsilon)$ consists of a truncated DVR $R$\footnote{P. Deligne requires $R$ to have a perfect residue field, but we will not include this in our definition.  Of course the relationship between valued hyperfields and triples proved in this paper will also hold if one includes this hypothesis in both the definition of valued hyperfields and that of triples}, a free $R$-module $M$ of rank 1, and a homomorphism $\epsilon:M\rightarrow R$ whose image is the maximal ideal $\mathfrak{m}_R$.  We define a integer valued function on $M^{\otimes i}$ by $v(am^{\otimes i})=i+v_R(a)$ for $a\in R$, where $m$ is a generator of M.\end{defn}

Note that for $s>r$, we can define a map $\epsilon_{r,s}:M^{\otimes s}\rightarrow M^{\otimes r}$ by $\epsilon_{r,s}(x^{\otimes s})=\epsilon(x)^{s-r}x^{\otimes r}$, where $x$ is a generator of $M$.  This map is used in the following definition of a morphism of triples. 

\begin{defn}[{\cite[1.4]{deligne}}]\label{tripmorph}  A morphism of triples $(r,f,\eta):(R,M,\epsilon)\rightarrow (R',M',\epsilon')$ consists of a homomorphism $f:R\rightarrow R'$, an integer $r$ (called the ramification index) and an $R$-linear map $\eta:M\rightarrow M'^{\otimes r}$, such that $f\epsilon=\epsilon_{0,r}'\eta$ and such that the induced map $M\otimes R' \rightarrow M'^{\otimes r}$ is an isomorphism of $R'$-modules.  We compose morphisms of triples by the formula $(r,f,\eta)(s,g,\theta)=(rs,fg,\eta^{\otimes s}\theta)$.\end{defn}

While P. Deligne had worked with triples, he had motivated his paper in terms of valued hyperfields instead.  He justified this by stating that his triples are essentially the same as Krasner's valued hyperfields.  Unfortunately, he did not give much indication as to how they are related, let alone a proof.  The goal of this paper is to understand the relation between these two notions.

One difference between the category of triples and the category of valued hyperfields is that the valuation on a triple is always discrete.  Thus we must restrict ourselves to working with discretely valued hyperfields.  Secondly, the category of discretely valued hyperfields contains not only the proper quotients $K/(1+\mathfrak{m}_K^i)$ of a local field $K$, but also $K$ itself.  On the other hand, there is no triple corresponding to the trivial quotient, so in order to relate discretely valued hyperfields with triples, we must exclude those discretely valued hyperfields which are fields.  One of the main results of this paper is the following.

\begin{thm}\label{mainthm1}There is a faithful essentially surjective functor $\mathrm{Tr}$ from the category of discretely valued hyperfields which are not fields to the category of triples.\end{thm}

Unfortunately, this functor is not an equivalence of categories because it is not full.  This problem is essentially due to the distinction between a discrete valuation and an absolute value.     To illustrate this distinction, suppose that $K$ is a local field and $\pi_K$ is a generator of its maximal ideal.  Then our convention is that $v(\pi_K)=1$, but on the other hand, we may choose the absolute value in such a way that $|\pi_K|$ will be any number strictly between 0 and 1.  This is an issue because triples only have the valuation $v$, while discretely valued hyperfields have an absolute value.  Hence an isomorphism class of discretely valued hyperfields contains some extra information represented by a number in the open interval $(0,1)$, which is not present in the associated triple.  However, this paper will prove the following theorem which says that any morphism of triples may be lifted to a morphism of discretely valued hyperfields after changing the definition of the absolute value on the target.  Morally this means that the aforementioned issue is the only obstacle preventing the functor from being full.

\begin{thm}Let $H_1,H_2$ be discretely valued hyperfields which are not fields.  Let $\mathrm{Tr}$ be the functor of Theorem \ref{mainthm1}.  Let $u:\mathrm{Tr}(H_1)\rightarrow \mathrm{Tr}(H_2)$ be a morphism of triples.  Then there is a valued hyperfield $H_2'$ with $H_2=H_2'$ as hyperfields, but with a different but equivalent absolute value, such that $\mathrm{Tr}(H_2)=\mathrm{Tr}(H_2')$ and such that $u:\mathrm{Tr}(H_1)\rightarrow \mathrm{Tr}(H_2')$ lifts to a morphism $H_1\rightarrow H_2'$.\end{thm}

The situation is particularly nice when one is interested only in the extensions of a particular discretely valued hyperfield or triple.  In this situation, the ambiguity about how to normalize the absolute value disappears and one has the following equivalence of categories.

\begin{thm}Let $H$ be a discretely valued hyperfield which is not a field.  Then $\mathrm{Tr}$ induces an equivalence of categories between the coslice category under $H$\footnote{Recall that the coslice category of a category $\mathcal{C}$ under an object $X\in\mathcal{C}$ is the category whose objects consist of objects $Y\in \mathcal{C}$ equipped with a morphism $X\rightarrow Y$ and whose morphisms $Y\rightarrow Z$ are morphisms in $\mathcal{C}$ such that the composite $X\rightarrow Y\rightarrow Z$ agrees with the morphism $X\rightarrow Z$ which is given as part of the structure of $Z$.  For example, the category of $R$-algebras is the coslice category of the category of rings under an object $R$.} and the coslice category under $\mathrm{Tr}(H)$.  It also induces an equivalence between the slice category over $H$ and the slice category over $\mathrm{Tr}(H)$.
\end{thm}

P. Deligne has defined notions of finiteness and flatness for morphisms of triples.  These notions were essential to his proof of an analogue of Theorem \ref{delhypthm} for triples, which involves interpreting the category of finite separable extensions $L$ of a local field $K$ which satisfy $\Gal(L/K)^i=1$ in terms of finite flat extensions of the associated triple.  The final section of this paper will show that these notions of finiteness and flatness have simple descriptions in terms of the associated hyperfields.

\section{Further Examples of Hyperfields}\label{examplessect}
Since this paper is likely to be the reader's first exposure to hyperfields, it is worth giving some more examples  so as to give the reader a feel for how they arise.  Most examples of hyperfields that arise in practice come from the following construction.

\begin{eg}Let $K$ be a field and $G\subseteq K^\times$ be a subgroup.  Then we define $K/G$ to be the set of orbits of $K$ under the action of $G$.  Since $K/G=\{0\}\cup K^\times/G$, it is a monoid under multiplication whose nonzero elements form a group.  Furthermore, one may define a multivalued addition operation by $xG+yG=\{zG\mid \exists g,h\in G \,\mathrm{such\, that}\, z\in xg+yh\}$.  It is a simple exercise to check that $K/G$ satisfies the hyperfield axioms.
\end{eg}

We have already seen this construction in the introduction to this paper, where it is applied to the case when $K$ is a local field and $G$ is a ball centered at $1$.  Our next example is the smallest possible hyperfield.  

\begin{eg}Let $\mathbb{K}=\{0,1\}$.  Define multiplication on $\mathbb{K}$ in the obvious way, and the multivalued addition operation by the equations $0+0=\{0\}$, $0+1=1+0=\{1\}$, and $1+1=\{0,1\}$.  Then $\mathbb{K}$ is a hyperfield called the \emph{Krasner hyperfield}.  If $K\neq\mathbb{F}_2$ is a field then $K/K^\times \cong \mathbb{K}$.  Consequently, $\mathbb{K}$ encodes the arithmetic of zero and non-zero numbers in the same way that $\mathbb{F}_2$ encodes the arithmetic of even and odd numbers.
\end{eg}

A result of Lyndon and Prenowitz allows one to interpret abstract projective spaces as vector spaces over $\mathbb{K}$ \cite{prenowitz}\cite{lyndon} (c.f. also \cite{adeleclasses}.)  It also turns out that the Zariski points of a scheme correspond bijectively to the $(\Spec\mathbb{K})$-valued points of the scheme\cite{monoidstohyper}.

\begin{eg}The \emph{hyperfield of signs} is defined as $\mathbb{S}=\{0,1,-1\}$.  The addition is defined by $1+1=\{1\}$, $-1+(-1)=\{-1\}$, $1+(-1)=\{0,1,-1\}$ and the equations $0+x=x+0=\{x\}$ for all $x\in\mathbb{S}$.  The multiplication is defined in the obvious way.  $\mathbb{S}$ is canonically isomorphic to the quotient $\mathbb{R}/\mathbb{R}_{>0}$ of the real numbers by the multiplicative action of the positive reals. Hence, one may interpret $\mathbb{S}$ as encoding the arithmetic of zero, positive and negative numbers.
\end{eg}

The hyperfield of signs has played a large role in the work of M. Marshall\cite{marshall}.  Motivated by the demands of real algebraic geometry, he had defined the abstract real spectrum of a ring $R$ as a topological space whose points correspond to pairs $(P,\leq)$ where 
$P\subseteq R$ is prime and $\leq$ is a relation on $R/P$ which makes it into a totally ordered ring.  He has shown that points of the abstract real spectrum of $R$ correspond to homomorphisms $R\rightarrow \mathbb{S}$, which we may think of more geometrically as the $(\Spec\mathbb{S})$-valued points of $\Spec R$.  A. Connes and C. Consani have computed the set of points of the abstract real spectrum of $ \mathbb{Z}[T]$, which turns out to be very similar to the set of real numbers\footnote{This stands in contrast to the fact that if $R$ is a ring then the $(\Spec R)$-valued points of $\Spec \mathbb{Z}[T]$ correspond to elements of $R$.}\cite{monoidstohyper}.

\begin{eg}Let $K$ be a field.  Then $K/(K^\times)^2$ is a hyperfield.  The addition may be described explicitly by stating that $[a]\in [c_1]+\ldots+[c_n]$ if and only if the element $a\in K$ may be represented by the diagonal quadratic form $c_1x_1^2+\ldots+c_nx_n^2$.\end{eg}

M. Marshall has shown, using the Milnor K-theory of hyperfields, that one has an isomorphism $W(K)\cong W(F)$ of Witt rings of quadratic forms for two fields $F$ and $K$, whose characteristic is not $2$, if and only if there is an isomorphism of hyperfields $K/(K^\times)^2\cong F/(F^\times)^2$\cite{marshallreview}.  Hence, the hyperfield $K/(K^\times)^2$ contains the same information as $W(K)$.

\begin{eg}Let $\mathbb{Y}=\mathbb{R}\cup \{-\infty\}$.  Define multiplication on $\mathbb{Y}$ as addition of real numbers.  Define a sum on $\mathbb{Y}$ by declaring the sum of $x,y\in \mathbb{Y}$ to be $\max(x,y)$ if $x\neq y$ and to be $\{t\in\mathbb{Y}\mid t\leq x\}$ if $x=y$.  Then one can check that $\mathbb{Y}$ is a hyperfield.\end{eg}

The hyperfield $\mathbb{Y}$ was introduced by O. Viro, who showed that a tropical subvariety of $\mathbb{R}^n$ may be interpreted as a zero set of a family of polynomials over $\mathbb{Y}$\cite{viro}.  He has also shown that multiplicative seminorms on a ring $R$ correspond bijectively to hyperring homomorphisms $R\rightarrow \mathbb{Y}$.

\begin{eg}Let $\mathcal{T}\mathbb{R}=\mathbb{R}$ as multiplicative monoids.  Define addition on $\mathcal{T}\mathbb{R}$ as follows.  If $|x|>|y|$ then $x+y=x$, while if $|y|>|x|$ then $x+y=y$.  If $x=y$ then $x+y=x$. If $x=-y$ then $x+y$ is the closed interval $[-x,x]$.  Then, $\mathcal{T}\mathbb{R}$ is a hyperfield.\end{eg}

There is a \emph{dequantization} process which allows one to construct the tropical semifield from the semifield of nonnegative real numbers.  O. Viro observed that by applying this same process to $\mathbb{R}$ instead of to $\mathbb{R}_{\geq 0}$, one obtains the hyperfield $\mathcal{T}\mathbb{R}$\cite{viro}.  A. Connes and C. Consani have reinterpreted this dequantization as the \emph{universal perfection} of the real numbers\cite{thickreals}.  This interpretation, together with an analogue of the Witt construction for $\mathcal{T}\mathbb{R}$, has allowed them to find archimedean analogues of several aspects of $p$-adic Hodge theory.

\section{Notation}

If $H$ is a discretely valued hyperfield, we let $\theta_H$ be the smallest element of $\{|x|\mid x\in H^\times\}$ which is less than $1$.  We define a map $v:H\rightarrow \mathbb{Z}\cup\{\infty\}$ by $v(x)=\frac{\log |x|}{\log \theta_H}$.  We say an element of $\pi\in H$ is a \emph{uniformizer} if $v(\pi)=1$. 

If $H$ is a valued hyperfield, we use $\rho_H$ to denote its norm (c.f. Definition \ref{defofvh}).  We will use $\mathcal{O}_H$ to denote the closed ball of radius $1$ centered at $0$ inside $H$, and $\mathfrak{m}_H^k$ to denote the closed ball of radius $\theta_H^k$ around $0$.

If $(R,M,\epsilon)$ is a triple we say $x\in M$ is a \emph{uniformizer} if it generates $M$.

If $R$ is a ring, we let $l(R)$ denote the length of $R$ when viewed as a module over itself.

\section{Construction of the triple \texorpdfstring{$\mathrm{Tr}(H)$}{Tr(H)}}\label{tripcons}
Let $H$ be a discretely valued hyperfield, which is not a field.  For $x,y\in H$ we write $x\equiv_{\eta} y$ when $d(x,y)\leq \eta$.  Define $M_i=\mathfrak{m}_H^i/\equiv_{\rho_H\theta_H^i}$.  Before studying the objects $M_i$ we make the following remark, which is a consequence of the discreteness of the absolute value.

\begin{rmk}\label{discvalrmk}
We remark that the distance between two distinct elements is always a power of $\theta_H$.  Let $B\subseteq H$ be an open or closed ball of radius $r>0$ centered at a point $x\in H$.  Let $i$ be such that $\theta_H^i\leq r<\theta_H^{i-1}$ if $B$ is a closed ball or $\theta_H^i<r\leq \theta_H^{i-1}$ if $B$ is an open ball.  Then $B$ is a closed ball of radius $\theta_H^i$, and this radius is minimal among all $u$ such that $B$ is a ball of radius $u$.  Now let $k$ be chosen either such that $\theta_H^k\leq \rho_H<\theta_H^{k-1}$ or such that $\theta_H^k< \rho_H\leq\theta_H^{k-1}$, depending on whether we are in the open or closed case of axiom (v) of Definition \ref{defofvh}.  Then the ball $x+y$ of radius $\rho_H\max(|x|,|y|)$ appearing in the definition is a closed ball of radius $\theta_H^k\max(|x|,|y|)$, and this radius is minimal.  By the minimality of $\rho_H$, this implies that $\rho_H=\theta_H^k$ is a power of $\theta_H$, and that the closed case in axiom (v) is always the relevant one\footnote{This is not true if we do not require the valuation to be discrete.}.\end{rmk}

We are now ready to study the quotients $M_i$.  It turns out that in this case the quotient construction collapses the multivalued addition on $\mathfrak{m}_H^i$ into a single valued operation.

\begin{lem}$M_i$ is an abelian group for all $i\in \mathbb{Z}$.  $M_0$ is a commutative ring, and each $M_i$ is a module over $M_0$.\end{lem}

\begin{proof}Let $x,y\in M_i$.  Let $\hat{x},\hat{y}\in\mathfrak{m}_H^i$ be lifts.  Let $\hat{z}\in \hat{x}+\hat{y}$.  Then $x+y\in M_i$ is defined to be it's equivalence class.  To show this is well-defined, let $\hat{x}',\hat{y}'\in\mathfrak{m}_H^i$ be another choice of lifts.  Then $|\hat{x}-\hat{x}'|\leq \rho_H\theta_H^i$ unless $0\in\hat{x}-\hat{x}'$.  On the other hand, if  $0\in\hat{x}-\hat{x'}$, then $\hat{x}-\hat{x}'$ is a ball around 0 of radius $\rho_H\mathrm{max}(\hat{x},\hat{x}')\leq \rho_H\theta_H^i$.  Thus we have $|\hat{x}-\hat{x}'|\leq \rho_H\theta_H^i$ in both cases, and similarly, $|\hat{y}-\hat{y}'|\leq \rho_H\theta_H^i$.  Let $\hat{z}'\in\hat{x}'+\hat{y}'$.  Then $\hat{z}-\hat{z}'\in (\hat{x}-\hat{x}')+(\hat{y}-\hat{y}')$, so $|\hat{z}-\hat{z}'|\leq\mathrm{max}(|\hat{x}-\hat{x}'|,|\hat{y}-\hat{y}'|)\leq \rho_H\theta_H^i$.  Thus $\hat{z}$ and $\hat{z}'$ define the same element of $M_i$.  Each of the abelian group axioms follows easily by using the corresponding facts in $\mathfrak{m}_H^i$.

We now define a bilinear multiplication map $M_i \times M_j\rightarrow M_{i+j}$.  Let $x\in M_i$ and $y\in M_j$.  Let $\hat{x}\in\mathfrak{m}_H^i$ and $\hat{y}\in\mathfrak{m}_H^j$ be lifts.  We define $xy\in M_{i+j}$ to be the class of $\hat{x}\hat{y}$.  Let $\hat{x}'$ be a different lift of $x$.  Then $d(\hat{x}\hat{y},\hat{x}'\hat{y})=|(\hat{x}-\hat{x'})\hat{y}|=|\hat{x}-\hat{x'}| |\hat{y}|\leq \rho_H\theta_H^i \theta_H^j$\footnote{The equalities here actually hold only when $d(\hat{x}\hat{y},\hat{x}'\hat{y})\neq 0$, but either way we get the inequality $d(\hat{x}\hat{y},\hat{x}'\hat{y})\leq \rho_H\theta_H^i \theta_H^j$.} since $\hat{y}\leq \theta_H^j$.  Thus $xy$ is independent of $\hat{x}$ and similarly it is independent of $\hat{y}$.  Bilinearity follows from the distributive law in $H$.  It is easy to check, using the associativity of $H$, that the multiplication $M_0\times M_0\rightarrow M_0$ makes $M_0$ into a ring, and that $M_0\times M_i\rightarrow M_i$ makes $M_i$ into a module.

\end{proof}

Henceforth we will denote $M_0$ by $R$ and $M_1$ by $M$.

\begin{lem}\label{Rtdvr}$R$ is a truncated DVR.  Its length is $\frac{\log \rho_H}{\log\theta_H}$.\end{lem}

\begin{proof}
For $x\in R$, let $\hat{x}\in\mathcal{O}_H$ be a lift.  Define $v(x)=v(\hat{x})$ if $x\neq 0$ and $v(0)=\infty$. To see this is well-defined, suppose $x\neq 0$, and let $\hat{x}'$ be another lift. Then $|\hat{x}'-\hat{x}|\leq \rho_H$, but $|\hat{x}|>\rho_H$.  By the ultrametric inequality, $|\hat{x}|=|\hat{x'}|$, so $v(x)$ is well-defined.  

For $x,y\in R$ such that $xy\neq 0$, $v(xy)=v(x)+v(y)$, as may be seen by picking lifts of $x$ and $y$.  In addition, $v(x+y)\geq \mathrm{min}(v(x),v(y))$.  Suppose that $x,y\in R$ are such that $v(x)\leq v(y)$.  Suppose $y\neq 0$, so that we also have $x\neq 0$.  Pick lifts $\hat{x},\hat{y}\in\mathcal{O}_H$.  Then $v(\hat{x})\leq v(\hat{y})$, so there is a $\hat{z}\in\mathcal{O}_H$ such that $\hat{y}=\hat{x}\hat{z}$.  Let $z\in R$ be the class of $\hat{z}$.  Then $y=xz$.  Of course if $y=0$ then we get a similar inequality by taking $z=0$.  We have shown that if $v(y)\geq v(x)$, then $y\in xR$.

Suppose $\pi_H\in H$ is such that $v(\pi_H)=1$.  Let $\pi\in R$ be its class.  Suppose for the moment that $\pi=0$.  Let $x\in R$ be nonzero, and let $\hat{x}\in \mathcal{O}_H$ be a lift.  Then there does not exist $y$ such that $x=\pi y$, so there does not exist $\hat{y}\in\mathcal{O}_H$ such that $\hat{x}=\pi_h\hat{y}$.  Hence $v(\hat{x})=0$ so that $v(x)=0$.  Hence $v(x)=0\leq 0=v(1)$ so that $x$ divides $1$ and $R$ is a field, and hence is a truncated DVR.  So in the case where $\pi=0$, we are done, and so we may suppose $\pi\neq 0$.

We now have $v(\pi)=1$.  Let $I$ be an ideal generated by a set $S$.  Let $i=\inf_{x\in S} v(x)$.  Then $S\subseteq \pi^i R$.  $\pi^i \in I$ because $S\subseteq I$ contains an element of valuation $i$.  Hence every ideal has the form $I=\pi^i R$, so $R$ is local and has a principal maximal ideal.  Since $\pi^\frac{\log \rho_H}{\log\theta_H}$ is the smallest power of $\pi$ which is $0$, $R$ is Artinian, and the assertion about the length holds.\footnote{This step is where we use the assumption that $H$ is not a field.  If $H$ were a field, then $\rho=0$, so the length would be infinite.  We would then have a DVR rather than a truncated DVR.  In fact, in this case, the construction described just gives the ring of integers.}
\end{proof}

We will denote the maximal ideal of $R$ by $\mathfrak{m}_R$.

\begin{lem}\label{Mitens} $M$ is free of rank 1.  Furthermore, there is a canonical isomorphism $M_i\cong M^{\otimes i}$ for $i\in \mathbb{N}$.\end{lem}

\begin{proof} Let $\pi\in H$ be a uniformizer.  Multiplication by $\pi$ gives a bijection $\mathcal{O}_H\rightarrow \mathfrak{m}_H$.  It is easily seen that this induces a well-defined bijection $\mathcal{O}_H/\equiv_{\rho_H}\rightarrow \mathfrak{m}_H/\equiv_{\theta_H\rho_H}$.  Since this bijection is just multiplication by  $\bar{\pi}\in M$, it is a homomorphism of modules, and so $M$ is free of rank 1.  A similar argument shows $M_i$ is free and generated by $\overline{\pi^i}$.  We define an isomorphism $M_i\cong M^{\otimes i}$ sending $\overline{\pi^i}$ to $\bar{\pi}^{\otimes i}$.  It is easy to check that this isomorphism is canonical in the sense that it is independent of the choice of $\pi$.\end{proof}

\begin{rmk}\label{negtenspow}Since $M$ is free of rank $1$, we can define $M^{\otimes k}$ for $k<0$ as well.  In this case one defines $M^{\otimes k}$ as the dual $\mathrm{Hom}(M^{\otimes -k},R)$ of $M^{\otimes -k}$.  If $\pi\in M$ is a generator, we obtain a generator $\pi^{\otimes k}\in M^{\otimes k}$ as the unique linear map sending $\pi^{\otimes -k}$ to $1$.  If $\pi'=u\pi\in M$ is another generator, we get $\pi'^{\otimes k}=u^k\pi^{\otimes k}$.  If $j$ and $k$ are arbitrary, it is a straightforward exercise to see that the map $M^{\otimes j}\otimes M^{\otimes k}\rightarrow M^{\otimes( j+k)}$ sending $\pi^{\otimes j}\otimes \pi^{\otimes k}$ to $\pi^{\otimes(j+k)}$ is a well defined isomorphism.  Furthermore, one may easily define an isomorphism $M^{\otimes (jk)}\cong (M^{\otimes j})^{\otimes k}$.  All of this is a standard part of the theory of line bundles.  The proof of Lemma \ref{Mitens} carries over easily to the case of negative tensor powers.  Another useful property of tensor powers is that given an isomorphism $\psi:X\rightarrow Y$ of free modules of rank $1$, which sends a generator $x\in X$ to a generator $y\in Y$, we can obtain a well-defined isomorphism $\psi^{\otimes k}:X^{\otimes k}\rightarrow Y^{\otimes k}$ sending $x^{\otimes k}$ to $y^{\otimes k}$.  However, if $\psi$ is not an isomorphism then a construction of this sort may only be done for nonnegative tensor powers.\end{rmk}

We now construct a map $\epsilon:M\rightarrow R$.  Let $x\in M$.  Let $\hat{x}\in \mathfrak{m}_H\subseteq \mathcal{O}_H$ be a lift.  Then $\epsilon(x)$ is defined to be the class of $\hat{x}$ in $R$.

\begin{lem} $\epsilon$ is a well defined $R$-linear map.  Furthermore, its image is $\mathfrak{m}_R$.
\end{lem}

\begin{proof}
Let $\hat{x},\hat{x}' \in\mathfrak{m}_H$ be lifts of $x\in R$.  Then $\hat{x}\equiv_{\theta_H\rho_H} \hat{x}'$, so $\hat{x}\equiv_{\rho_H} \hat{x}'$.  Thus they give the same element of $R$, and so $\epsilon$ is well-defined.  The $R$-linearity is trivial.   Because the map is $R$-linear and because $M$ is free of rank $1$, we may describe its image by computing what it does to a generator of $M$.  If we let $\pi_H$ be an element of $H$ with $v(\pi_H)=1$, then $M$ is generated by the class of $\pi_H\in\mathfrak{m}_H$ while $\mathfrak{m}_R$ is the principal ideal generated by the class of $\pi_H\in\mathcal{O}_H$, so we see that the image is as described.
\end{proof}

\begin{defn} $\mathrm{Tr}(H)=(R,M,\epsilon)$.\end{defn}

We have proven the following theorem.

\begin{thm}$\mathrm{Tr}(H)$ is a triple in the sense of Deligne.\end{thm}

As an example, we will explicitly compute what the constructions of this section give in the case where $H=K/(1+\mathfrak{m}_K^i)$ is the hyperfield that was considered in the introduction.

\begin{eg}Let $K$ be a local field and $H=K/(1+\mathfrak{m}_K^i)$.  One has $\rho_H=\theta_H^i$.  For $k\in \mathbb{Z}$, consider the composite of the quotient maps $\alpha:\mathfrak{m}_K^k\rightarrow \mathfrak{m}_H^k$ and $\beta: \mathfrak{m}_H^k\rightarrow M_k$.  One has $\alpha(x+y)\in \alpha(x)+\alpha(y)$ for all $x,y\in   \mathfrak{m}_K^k$ and $\beta(x+y)=\beta(x)+\beta(y)$ for all $x,y\in \mathfrak{m}_H^k$.  Hence $\beta(\alpha(x+y))=\beta(\alpha(x))+\beta(\alpha(y))$ for all $x,y\in\mathfrak{m}_K^k$, and in fact $\beta\alpha$ is a homomorphism of $\mathcal{O}_K$-modules.  By definition of $\beta$, one sees $\ker (\beta\alpha)$ contains all elements with absolute value at most $\rho_H\theta_H^k=\theta_H^{k+i}$.  Using the fact that $\alpha$ preserves absolute value, one can see that $\ker (\beta\alpha)$ does not contain an element of absolute value $\theta_H^{k+i+1}$.  Hence $\ker(\beta\alpha)=\mathfrak{m}_K^{k+i}$ so $M_k\cong \mathfrak{m}_K^{k}/\mathfrak{m}_K^{k+i}$.  In particular $\mathrm{Tr}(K/(1+\mathfrak{m}_K^i))=(\mathcal{O}_K/\mathfrak{m}_K^i,\mathfrak{m}_K/\mathfrak{m}_K^{1+i},\epsilon)$, for some map $\epsilon$.  One may easily compute that $\epsilon:\mathfrak{m}_K/\mathfrak{m}_K^{1+i}\rightarrow\mathcal{O}_K/\mathfrak{m}_K^i$ is the map induced by the inclusion $\mathfrak{m}_K\subseteq \mathcal{O}_K$.\end{eg}

\section{Functoriality}\label{func}
Let $H,H'$ be discretely valued hyperfields, which are not fields.  We will retain all the notation of the previous section.  In addition we will define $\epsilon'$, $R'$, $M'$, and $M'_i$ in a manner analogous to that of the previous section, but using $H'$ instead of $H$.\footnote{So for example, $\mathrm{Tr}(H')=(R',M',\epsilon')$.}  Throughout this section, we let $f:H\rightarrow H'$ be a morphism of valued hyperfields.  We will let $r=\frac{\log \theta_H}{\log \theta_{H'}}$.

The following lemma is due to M. Krasner.

\begin{lem}\label{increasingrho}\cite[pg149]{krasner}$\rho_{H'}\geq\rho_H$.\end{lem}




We also need the following lemma.

\begin{lem}Let $f:H\rightarrow H'$ be a morphism of valued hyperfields, and let $x,y\in H$.  Then $d(f(x),f(y))\leq d(x,y)$, with equality if $f(x)\neq f(y)$.\end{lem}
\begin{proof}If $f(x)=f(y)$ we are done, so we may assume that $f(x)\neq f(y)$ and hence that $x\neq y$.  Let $z\in x-y$.  Then we have $|f(z)|=|z|=d(x,y)$.  On the other hand, $f(z)\in f(x)-f(y)$.  Since $f(x)\neq f(y)$ we have $d(f(x),f(y))=|f(z)|=d(x,y)$, as desired.\end{proof}

We define a map $\phi:R\rightarrow R'$ by letting $\phi(x)$ be the class of $f(\hat{x})$ where $\hat{x}\in H$ is any lift.

\begin{prop}\label{phiwelldef}$\phi$ is a well-defined ring homomorphism.\end{prop}

\begin{proof}Let $\hat{x},\hat{x}'\in H$ be lifts of $x$.  Then $\hat{x}\equiv_{\rho_H} \hat{x}'$, so $f(\hat{x})\equiv_{\rho_H} f(\hat{x}')$.  Then $f(\hat{x})$ and $f(\hat{x'})$ define the same class in $R'$ by Lemma \ref{increasingrho}.  Thus $\phi$ is well-defined.  Let $x,y\in R$, and let $\hat{x},\hat{y}$ be lifts.  Then any element $\hat{z}\in\hat{x}+\hat{y}$ is a lift of $x+y$.  Then $f(\hat{z})\in f(\hat{x}+\hat{y}) \subseteq f(\hat{x})+f(\hat{y})$, so the class of $f(\hat{z})$ is $\phi(x)+\phi(y)$.  Hence $\phi(x+y)=\phi(x)+\phi(y)$.  The other axioms of a ring homomorphism are easy to verify.
\end{proof}

We will now define a map $\eta:M\rightarrow M'^{\otimes r}\cong M'_r$.  For $x\in M$, we pick a lift $\hat{x}\in\mathfrak{m}_H$.  Then $f(\hat{x})\in\mathfrak{m}_{H'}^r$, and we let $\eta(x)$ be the element of $M'^{\otimes r}$ corresponding to the class of $f(\hat{x})$ in $M'_r$.

\begin{lem}$\eta$ is a well-defined $R$-linear map.  It induces an isomorphism $M\otimes R'\rightarrow M'^{\otimes r}$.\end{lem}

\begin{proof}This is proven in the same manner as Proposition \ref{phiwelldef}.   If we let $\hat{x}'$ be another lift of $x$, then since $\hat{x}\equiv_{\theta_H\rho_H}\hat{x}'$, $f(\hat{x})\equiv_{\rho_{H'}\theta_H} f(\hat{x'})$.  Since $\theta_H=\theta_{H'}^r$, we see that $f(\hat{x})\equiv_{\rho_{H'}\theta_{H'}^r}$, so that $f(\hat{x})$ and $f(\hat{x}')$ define the same element of $M'_r$.  $R$-linearity is straightforward to verify.  Let $\pi\in H$ be a uniformizer.  $M\otimes R'$ is free with generator $\bar{\pi}$, while $M'_r$ is free with generator $\overline{f(\pi)}=\eta(\bar{\pi})$.  Thus $M\otimes R'\rightarrow M'^{\otimes r}$ maps a generator to a generator, so is an isomorphism.\end{proof}

Deligne defined an $R'$-linear map $\epsilon'_{0,r}:M'^{\otimes r}\rightarrow R'$ by $\epsilon'_{0,r}(x^{\otimes r})=\epsilon(x)^r$ when $x$ generates $M'$.  It is straightforward to verify that for $x\in M'^{\otimes r}$, $\epsilon'_{0,r}(x)$ is the class of $\hat{x}$ in $R'$ where $\hat{x}\in \mathfrak{m}_{H'}^r\subseteq \mathcal{O}_{H'}$ is any lift of the element of $M'_r$ which corresponds to $x$.

\begin{lem} $\epsilon'_{0,r}\eta=\phi\epsilon$.\end{lem}
\begin{proof}

By linearity it suffices to prove this when $x\in M$ is the class of a uniformizer $\pi_H\in H$.  We pick a uniformizer $\pi_{H'}$ of $H'$ and write $f(\pi_H)=u\pi_{H'}^r$ with $u\in H'$.  $\eta(x)$ is then the class of $u\pi_{H'}^r$ in the quotient $M'^{\otimes r}\cong M'_r$ of $\mathfrak{m}_{H'}^r$.  Hence $\epsilon'_{0,r}(\eta(x))=\bar{u}\bar{\pi_{H'}}^r=\overline{f(\pi_H)}$ where $\bar{u}$, $\bar{\pi_{H'}}$, and $\overline{f(\pi_H)}$ represent classes in $R'$.  On the other hand, $\epsilon(x)$ is the class of $\pi_H$, so  $\phi\epsilon(x)$ is the class of $f(\pi_H)$.  Hence both maps agree for this choice of $x$ and hence the maps are equal.
\end{proof}

\begin{defn}\label{trfdef} $\mathrm{Tr}(f)$ will denote $(r,\phi,\eta)$ where $r$, $\phi$, and $\eta$ are as above.\end{defn}

We have proven the following theorem.

\begin{thm} $\mathrm{Tr}(f)$ is a morphism of triples.\end{thm}

\begin{thm} $\mathrm{Tr}$ is a functor from the category of discretely valued hyperfields which are not fields to the category of triples.  \end{thm}
\begin{proof}
We only need to show it is compatible with composition.  That is, if $f:H\rightarrow H'$ and $f':H'\rightarrow H''$ and $\mathrm{Tr(f')}=(r',\phi',\eta')$, then we need to show $\mathrm{Tr}(f'f)=(rr',\phi'\phi,\eta'^{\otimes r}\eta)$.  Let $(\hat{r},\hat{\phi},\hat{\eta})=\mathrm{Tr}(f'f)$.  The claim about ramification indices follows from $r=\frac{\log \theta_H}{\log\theta_{H'}}$ and $r'=\frac{\log \theta_{H'}}{\log\theta_{H''}}$.

Let $x\in\mathcal{O}_H$.  Let $\bar{x}\in R$ be its class.  Then $\phi(\bar{x})$ is the class of $f(x)$.  If $x'\in \mathcal{O}_{H'}$ and if $\bar{x}'\in R'$ is its class then $\phi'(\bar{x}')$ is the class of $f'(x')$.  Applying this to $x'=f(x)$, we see that $\phi'(\phi(\bar{x}))$ is the class of $f'(f(x))$.  Since $\hat{\phi}(\bar{x})$ is the class of $f'(f(x))$ we get $\hat{\phi}=\phi'\phi$.

Let $\pi''\in H''$, $\pi'\in H'$, and $\pi\in H$ be uniformizers.  Let $u''\in H''$ and $u'\in H'$ be such that $f'(\pi')=u''\pi''^{r'}$ and $f(\pi)=u'\pi'^r$.  Then $(f'\circ f)(\pi)=f'(u')u''^{r}\pi^{rr''}$.  Let $x\in M$, $x'\in M'$ and $x''\in M''$ be the classes of $\pi$, $\pi'$, and $\pi''$.  Then $\eta(x)$ is the class of $f(\pi)$ so $\eta(x)=\bar{u}'(x')^{\otimes r}$.  Similarly $\eta'(x')=\bar{u}''(x'')^{\otimes r'}$.  Hence $\eta'^{\otimes r}(\eta(x))=\phi'(\bar{u}')\eta'^{\otimes r}(x'^{\otimes r})=\phi'(\bar{u}')\bar{u}''^{r}x''^{\otimes rr'}$, which is the class of $f'(u')u''^r\pi''^{rr'}=f'(f(\pi))$, which in turn is $\hat{\eta}(x)$.  Since both maps are $R$-linear and since $x$ generates $M$, we have $\hat{\eta}=\eta'^{\otimes r}\eta$.
\end{proof}

\section{Recovering the underlying set of the hyperfield}
Let $T=(R,M,\epsilon)$ be any triple.  We define $v:M^{\otimes i}\rightarrow \mathbb{Z}\cup\{\infty\}$\footnote{Since $M$ is projective of rank 1, negative tensor powers are defined, as in Remark \ref{negtenspow}.} by $v(r\pi^{\otimes i})=v(r)+i$ for $r\in R$, when $\pi$ is a uniformizer.  Let $\mathbf{U}(T)=\{0\}\cup\displaystyle\bigcup_{i\in\mathbb{Z}} \{x\in M^{\otimes i}\mid v(x)=i\}$.  If $(r,\phi,\eta):(R,M,\epsilon)\rightarrow (R',M',\epsilon')$ is a morphism of triples, then it induces maps $\eta^{\otimes i}:M^{\otimes i}\rightarrow M^{\otimes ri}$ which send elements of valuation $i$ to those of valuation $ri$.  These give a map $\mathbf{U}(r,\phi,\eta):\mathbf{U}(R,M,\epsilon)\rightarrow \mathbf{U}(R',M',\epsilon)$.  It is readily verified that $\mathbf{U}$ is a functor.

\begin{prop}\label{forget}$\mathbf{U}\circ\mathrm{Tr}$ is naturally isomorphic to the forgetful functor from the category of discretely valued hyperfields which are not fields to the category of sets.\end{prop}

\begin{proof}Let $H$ be a discretely valued hyperfield which is not a field.  Let $T=\mathrm{Tr}(H)=(R,M,\epsilon)$.  Let $M_i$ be as in \S\ref{tripcons}.  Let $C_i=\{x\in H\mid v(x)=i\}$.  Suppose $x\in C_i$ and $y\in H$ are chosen such that $x\equiv_{\theta_H^i\rho_H} y$.  Then by page 145 of Krasner, $x=y$.  Thus the reduction map $C_i\rightarrow M_i$ is injective.   Its image consists of elements with valuation $i$, so we have bijections $C_i\xrightarrow{\alpha_i}\{x\in M_i\mid v(x)=i\}\xrightarrow{\beta_i} \{x\in M^{\otimes i}\mid v(x)=i\}$.  We wish to show these bijections are natural in the sense that the following diagram commutes for any morphism $f:H\rightarrow H'$ of valued hyperfields, where the vertical arrows are the maps induced by $f$, and where we put an apostrophe next to the name of a construction to indicate that the construction is done using $H'$ rather than $H$.\\

$\begin{CD}
C_i @>\alpha_i>> \{x\in M_i\mid v(x)=i\}@>\beta_i>>\{x\in M^{\otimes i}\mid v(x)=i\}\\
@VfVV @V\theta_iVV @V\eta^{\otimes i}VV\\
C'_{ri} @>\alpha'_{ri}>> \{x\in M'_{ri}\mid v(x)=ri\}@>\beta'_{ri}>>\{x\in M'^{\otimes ri}\mid v(x)=ri\}
\end{CD}$\\\

Let $f:H\rightarrow H'$ be a morphism to another discretely valued hyperfield which is not a field.  Let $\mathrm{Tr}(H')=(R',M',\epsilon')$, and let $C_i'$ and $M_i'$ be like $C_i$ and $M_i$, but defined in terms of $H'$ instead of $H$.  Let $(r,\phi,\eta)=\mathrm{Tr}(f)$.  Let $x\in C_i$.  Then $\alpha_{ri}(f(x))$ is the reduction of $f(x)$ modulo $\equiv_{\theta_{H'}^i\rho_{H'}}$.  Let $\theta_i:M_i\rightarrow M'_{ri}$ be the map corresponding to $\eta^{\otimes i}$.  It is routine to verify that $\theta_i(x)$ is obtained by lifting $x$, applying $f$, and reducing.  Then $\theta_i(\alpha_i(x))$ is obtained by reducing $x$, picking a lift, applying $f$ to that lift, and reducing again.  Thus $\theta_i(\alpha_i(x))=\alpha_{ri}(f(x))$, so the left square commutes.  The right square commutes by the choice of $\theta_i$.  Thus the bijections describing the horizontal arrows of the above diagram are natural.  Hence, so is the induced bijection $H=\{0\}\cup\displaystyle\bigcup_{i\in\mathbb{Z}} C_i\rightarrow \{0\}\cup\displaystyle\bigcup_{i\in\mathbb{Z}} \{x\in M^{\otimes i}\mid v(x)=i\}=\mathbf{U}(\mathrm{Tr}(H))$, and the result follows.
\end{proof}

\begin{cor}$\mathrm{Tr}$ is faithful.
\end{cor}

\begin{proof}This follows from Proposition \ref{forget} and the fact that the forgetful functor is faithful.\end{proof}

\section{Equivalence}

Let $H$ be a discretely valued hyperfield which is not a field.  We have seen that there is a canonical bijection $\psi:\mathbf{U}(\mathrm{Tr}(H))\rightarrow H$, so $\widetilde{H}=\mathbf{U}(\mathrm{Tr}(H))$ is a discretely valued hyperfield\footnote{By decreeing $\psi$ to be an isomorphism.}.  We will now describe the addition, multiplication, and absolute value on $\widetilde{H}$ more explicitly.  We will retain the notation of the previous section.  Let $S_i=\{x\in M^{\otimes i}\mid v(x)=i\}$, so $\widetilde{H}=\{0\}\cup\bigcup_i S_i$.  Let $\pi_H$ be a uniformizer in $H$, and $\pi_M$ be its image in $M$ (which must generate $M$).  Throughout this section, we will identify $M_i$ with $M^{\otimes i}$.

For $x\in S_i$, it follows from results of the previous section that $|\psi(x)|=\theta_H^i$, so $|x|=\theta_H^i$.  For $x\in S_i$ and $y\in S_j$, we can easily verify that $xy\in S_{i+j}\subseteq M^{\otimes i+j}$ is the image of $x\otimes y$ under $M^{\otimes i}\otimes M^{\otimes j}\rightarrow M^{\otimes i+j}$.  

Let $x\in S_j$ and $y\in S_i$.  Without loss of generality, we assume $i\geq j$.  Let $z=x+_{M_j}\epsilon_{j,i}(y)\in M_j$\footnote{We use the notation $+_{M_j}$ to distinguish this addition from the addition $+_{\widetilde{H}}$ which comes from the hyperfield structure on $\widetilde{H}$.}, where $\epsilon_{j,i}:M^{\otimes i}\rightarrow M^{\otimes j}$ is the map induced by $\epsilon$.  Let $\hat{x},\hat{y}\in H$ be lifts of $x\in M_j$ and $y\in M_i$. Note that $\hat{y}$ is also a lift of $\epsilon_{j,i}(y)$.  Then $z$ is by definition the reduction of any element of $\hat{x}+\hat{y}$.  Since $|y|\leq |x|=\theta_H^j$, $\hat{x}+\hat{y}$ is a ball of radius $\rho_H\theta_H^j$, so it is the preimage of $z$ under the reduction map.  Let $w\in \widetilde{H}$.  It is easy to check that $\psi(w)\in H$ reduces to $z\in M_j$ if and only if either both $w=0$ and $z=0$ hold or if $w\in S_k$ for some $k\geq j$ and $\epsilon_{j,k}(w)=z$, because any element of $H$ corresponding to $w\in M_k$ reduces to $\epsilon_{j,k}(w)\in M_j$.  Thus $x+_{\widetilde H} y=\bigcup_{k\geq j} \{w\in S_k\mid \epsilon_{j,k}(w)=x+\epsilon_{j,i}(y)\}$, or it is the union of this set with $\{0\}$ depending on whether $x=-y$.

Before reconstructing morphisms of discretely valued hyperfields from morphisms of triples, we need the following lemma.

\begin{lem}\label{nonobviouslem}Let $(r,\phi,\eta):(R,M,\epsilon)\rightarrow (R',M',\epsilon')$ be a morphism of triples.  Let $k\geq j$.  Then $\epsilon'_{rj,rk}\eta^{\otimes k}=\eta^{\otimes j}\epsilon_{j,k}$.\end{lem}
\begin{proof}Let $x'$ generate $M'$, and let $\pi'=\epsilon'(x')$.  Let $x$ generate $M$ and let $\pi=\epsilon(x)$.  Since $x$ generates $M\otimes R'$, $\eta(x)$ generates $M'^{\otimes r}$ as an $R'$-module. $x'^{\otimes r}$ also generates $M'^{\otimes r}$.  Hence we can pick $u\in R'^{\times}$ such that $\eta(x)=ux'^{\otimes r}$.  Now $\eta^{\otimes k}(x^{\otimes k})=u^kx'^{\otimes kr}$.  Then $\epsilon'_{rj,rk}(\eta^{\otimes k}(x^{\otimes k}))=\pi'^{(k-j)r}u^kx'^{\otimes jr}$.  On the other hand, $\epsilon_{j,k}(x^{\otimes k})=\pi^{k-j}x^{\otimes j}$ so $\eta^{\otimes j}(\epsilon_{j,k}(x^{\otimes k}))=u^j\phi(\pi)^{(k-j)}x'^{\otimes jr}$.  Thus the identity we wish to prove will follow if we can establish that $u^j\phi(\pi)^{k-j}x'^{\otimes jr}=\pi'^{(k-j)r}u^kx'^{\otimes jr}$.  By definition \ref{tripmorph}, we have $\phi(\pi)=\phi(\epsilon(x))=\epsilon'_{0,r}(\eta(x))=\epsilon'_{0,r}(ux'^{\otimes r})=u\pi'^{\otimes r}$.  Raising this to the $(k-j)$-th power and multiplying by $u^jx'^{\otimes jr}$ yields $u^j\phi(\pi)^{k-j}x'^{\otimes jr}=\pi'^{(k-j)r}u^kx'^{\otimes jr}$.  Hence $\eta^{\otimes j}(\epsilon_{j,k}(x^{\otimes k}))=u^j\phi(\pi)^{(k-j)}x'^{\otimes jr}=\pi'^{(k-j)r}u^kx'^{\otimes jr}=\epsilon'_{rj,rk}(\eta^{\otimes k}(x^{\otimes k}))$ so that $\epsilon'_{rj,rk}\eta^{\otimes k}$ and $\eta^{\otimes j}\epsilon_{j,k}$ agree on a generator, and hence are equal.
\end{proof}

Let $H,H'$ be discretely valued hyperfields which are not fields.  Let $(r,\phi,\eta):\mathrm{Tr}(H)\rightarrow \mathrm{Tr}(H')$ be a morphism of triples.  Let $f=\mathbf{U}(r,\phi,\eta):\widetilde{H}\rightarrow \widetilde{H'}$.  We would like to show that $f$ is a morphism of valued hyperfields, but this is not necessarily true.  Let $\pi\in H$ be a uniformizer.  If $f$ were a morphism of valued hyperfields, we would have $|f(\pi)|=|\pi|=\theta_H$ so that $v(f(\pi))=\frac{\log \theta_H}{\log\theta_{H'}}$.  However, we instead have $v(f(\pi))=r$.

\begin{prop}\label{hfmorphcons}If $r=\frac{\log \theta_H}{\log\theta_{H'}}$, then $f$ is a morphism of valued hyperfields.\end{prop}
\begin{proof}
By construction, $f$ maps elements of $S_i$ to elements of $S'_{ri}$ (via the maps $\eta^{\otimes i}$), and $r=\frac{\log \theta_H}{\log\theta_{H'}}$, so $f$ preserves absolute value.  $f$ preserves multiplication, because the multiplication is defined in terms of $M^{\otimes i}\otimes M^{\otimes j}\rightarrow M^{\otimes i+j}$ and the corresponding maps for $M'$, and because the maps $S_i\rightarrow S'_{ri}$ are just $\eta^{\otimes i}$.  Let $x\in \widetilde{H}$ be such that $f(x)=1$.  Then $x\in S_0\subseteq R$, and $\phi(x)=1$, so $x-1\in \ker(\phi)$.  Conversely, if $x-1\in \ker(\phi)$, then $x\in S_0$ and $f(x)=1$ when we view $x$ as an element of $\widetilde{H}$.  But the preimage of $\ker(\phi)$ (or of any other ideal of $R$) in $\mathcal{O}_H$ is a ball around $0$.  Hence the equation $f(x)=1$ is equivalent to a bound on $d(1,x)=|x-1|$, so the fiber of $1$ is a ball.  Consequently all fibers are balls.

Let $x,y\in \widetilde{H}$.  We wish to show that $f^{-1}(f(x)+_{\widetilde{H}'}f(y))=f^{-1}(f(x))+_{\widetilde{H}}f^{-1}(f(y))$.  Let $z$ be such that $f(z)\in f(x)+_{\widetilde{H'}}f(y)$.  For simplicity we will consider only the case where $z\neq 0$; the other case is trivial.  Suppose $i=v(y)\geq v(x)=j$.  Let $k=v(z)$.  Then $v(f(z))=rk$, and similarly for $x$ and $y$.  Then $\epsilon'_{rj,rk}(f(z))=f(x)+\epsilon'_{rj,ri}(f(y))$, by the description of $+_{\widetilde{H'}}$.  So $\epsilon'_{rj,rk}(\eta^{\otimes k}(z))=\eta^{\otimes j}x+\epsilon'_{rj,ri}(\eta^{\otimes i}(y))$.  Using Lemma \ref{nonobviouslem}, $\eta^{\otimes j} (\epsilon_{j,k}(z))=\eta^{\otimes j} (x)+\eta^{\otimes j} (\epsilon_{j,i}(y))$.  Let $x'=\epsilon_{j,k}(z)-\epsilon_{j,i}(y)$.  Since $f(x)=\eta^{\otimes j}(x)=\eta^{\otimes j} (\epsilon_{j,k}(z)-\epsilon_{j,i}(y))$, we have $f(x')=f(x)$ and $v(x')=v(x)=j$\footnote{To see that $x'\in S_j$, one notes that $rv(x)=v(\eta^{\otimes j}(x))=v(\eta^{\otimes j}(x'))=rv(x')$ so $v(x)=v(x')$.}.  Since $\epsilon_{j,k}(z)=x'+\epsilon_{j,i}(y)$, we have $z\in x'+_{\widetilde{H}}y$.  We then have $z\in f^{-1}(f(x))+_{\widetilde{H}}f^{-1}(f(y))$.  Hence $f^{-1}(f(x)+_{\widetilde{H'}}f(y))\subseteq f^{-1}(f(x))+_{\widetilde{H}}f^{-1}(f(y))$.

For the reverse inclusion, suppose $x,y\in \widetilde{H}$ with $j=v(x)\geq v(y)=i$, and let $z\in f^{-1}(f(x))+_{\widetilde{H}}f^{-1}(f(y))$.  We may pick $x',y'\in\widetilde{H}$ such that $z\in x'+_{\widetilde{H}} y'$, $f(x')=f(x)$, and $f(y')=f(y)$.  Since $f$ preserves absolute value, $v(x')=v(x)$ and $v(y')=v(y)$.  Let $k=v(z)$.  Since $z\in  x'+_{\widetilde{H}} y'$, we have $\epsilon_{j,k}(z)=x'+\epsilon_{j,i}(y')$ so $\eta^{\otimes j}(\epsilon_{j,k}(z))=\eta^{\otimes j}(x')+\eta^{\otimes j}(\epsilon_{j,i}(y'))$.  Then $\epsilon'_{rj,rk}(f(z))=\epsilon'_{rj,rk}(\eta^{\otimes k}(z))=\eta^{\otimes j}(x')+\epsilon'_{rj,ri}(\eta^{\otimes i}(y'))=f(x')+\epsilon'_{rj,ri}(f(y'))=f(x)+\epsilon'_{rj,ri}(f(y))$.  Then $f(z)\in f(x)+_{\widetilde{H}'}f(y)$.  Hence $z\in f^{-1}( f(x)+_{\widetilde{H}'}f(y))$ as desired.  Hence $f^{-1}(f(x)+_{\widetilde{H'}}f(y))= f^{-1}(f(x))+_{\widetilde{H}}f^{-1}(f(y))$.
\end{proof}

Note that since $H\cong \widetilde{H}$, we have $\mathrm{Tr}(H)\cong \mathrm{Tr}(\widetilde{H})$. 

\begin{cor}\label{fullcor}Let $(r,\phi,\eta):\mathrm{Tr}(H)\rightarrow \mathrm{Tr}(H')$ be a morphism of triples such that $r=\frac{\log \theta_H}{\log\theta_{H'}}$.  Then there is a morphism $f:H\rightarrow H'$ such that $(r,\phi,\eta)=\mathrm{Tr}(f)$.\end{cor}
\begin{proof}

We have a canonical isomorphism $H\cong \widetilde{H}$.  For $x\in H$ we write $\widetilde{x}$ for the corresponding element of $\widetilde{H}$, and we use similar notation for elements of $\widetilde{H}'$ as well.  Let $\widetilde{f}:\widetilde{H}\rightarrow \widetilde{H'}$ be given by $\widetilde{f}=\mathbf{U}(r,\phi,\eta)$, and let $f:H\rightarrow H'$ be obtained by composing $\widetilde{f}$ with the isomorphisms $H\cong \widetilde{H}$ and $H'\cong \widetilde{H}'$.  Let $(\hat{r},\hat{\phi},\hat{\eta})=\mathrm{Tr}(\widetilde{f})$.  It is easy to check $r=\hat{r}$.

Now let $x\in H$ have $v(x)=j\geq i$.  One easily sees that $\widetilde{f(x)}=\widetilde{f}(\widetilde{x})=\eta^{\otimes j}(\widetilde{x})$.  Let $u\in M^{\otimes i}\cong M_i$ be the class of $x\in \mathfrak{m}_H^i$.  Then $\widetilde{x}\in S_j\subseteq M^{\otimes j}$ satisfies $\epsilon_{i,j}(\widetilde{x})=u$.  Now $\eta^{\otimes i}(u)=\eta^{\otimes i}(\epsilon_{i,j}(\widetilde{x}))=\epsilon'_{ri,rj}(\eta^{\otimes j}(\widetilde{x}))=\epsilon'_{ri,rj}(\widetilde{f(x)})$.  Thus $\eta^{\otimes i}$ sends $u\in M^{\otimes i}\cong M_i$ (which is the class of $x\in\mathfrak{m}_H^i$) to $\epsilon'_{ri,rj}(\widetilde{f(x)})\in M'^{\otimes ri}\cong M'_{ri}$ (which is the class of $f(x)\in\mathfrak{m}_{H'}^{ri}$) whenever $i\leq v(x)$.  Since $\eta^{\otimes 0}:R=M^{\otimes 0}\rightarrow M'^{\otimes 0}=R'$ is just $\phi$, we see that for any $x\in\mathcal{O}_H$, $\phi$ sends the class of $x$ to the class of $f(x)$. Hence $\phi=\hat{\phi}$.  Applying the result instead to $\eta^{\otimes 1}=\eta$ shows that for any $x\in\mathfrak{m}_K$, $\eta$ sends the class of $x$ to the class of $f(x)$.  Hence $\eta=\hat{\eta}$.  Thus $\mathrm{Tr}(f)=(r,\phi,\eta)$.
\end{proof}

\begin{rmk}\label{rescaleabs}If $(r,\phi,\eta):\mathrm{Tr}(H)\rightarrow \mathrm{Tr}(H')$ is any morphism of triples, then we may define a new valued hyperfield $\mathring{H}'$ by setting $\mathring{H}'=H'$ as hyperfields but rescaling the absolute value in such a way that $\theta_{\mathring{H}'}=\theta_H^{1/r}$.  Note that $\mathrm{Tr}(H')=\mathrm{Tr}(\mathring{H}')$.  Then the morphism $(r,\phi,\eta):\mathrm{Tr}(H)\rightarrow \mathrm{Tr}(\mathring{H}')$ may be lifted to a morphism $H\rightarrow \mathring{H}'$.  Hence any morphism of triples lifts to a morphism of discretely valued hyperfields as long as we are willing to replace the absolute value on the target with a different but equivalent absolute value.\end{rmk}

\begin{thm}$\mathrm{Tr}$ is essentially surjective.\end{thm}
\begin{proof}Let $(R,M,\epsilon)$ be a triple.  Deligne has shown that for any truncated DVR $R$, there is a DVR $\mathcal{O}$ such that $R\cong \mathcal{O}/\mathfrak{m}^i$ for some $i$\cite[pg 126]{deligne}.  Let $K=\mathrm{Frac}(\mathcal{O})$.  Let $H=K/(1+\mathfrak{m}^i)$.  Let $(R',M',\epsilon')=\mathrm{Tr}(H)$.  Then $R'\cong \mathcal{O}/\mathfrak{m}^i\cong R$.  Deligne showed that an isomorphism of truncated DVRs extends to an isomorphism of triples\cite[pg 126]{deligne}; hence $\mathrm{Tr}(H)=(R',M',\epsilon')\cong (R,M,\epsilon)$.  Thus $\mathrm{Tr}$ is essentially surjective.
\end{proof}

\begin{thm}Let $H$ be a discretely valued hyperfield which is not a field.  Then $\mathrm{Tr}$ induces an equivalence of categories between the coslice category of $H$ and the coslice category of $\mathrm{Tr}(H)$.  It also induces an equivalence between the slice category of $H$ and the slice category of $\mathrm{Tr}(H)$.
\end{thm}

\begin{proof}We consider the case of the coslice category; the other part is proven similarly.  Clearly this functor is faithful.  Let $(r,\phi,\eta):\mathrm{Tr}(H)\rightarrow S$ be an object of the coslice category of $\mathrm{Tr}(H)$.  Then there exists $H'$ such that $S\cong \mathrm{Tr}(H')$.  By rescaling the absolute value on $H'$ (which does not affect $\mathrm{Tr}(H')$), we may assume $r=\frac{\log \theta_H}{\log\theta_{H'}}$.  Then there is a morphism $f:H\rightarrow H'$ such that $(r,\phi,\eta)=\mathrm{Tr}(f)$.  Hence the functor between coslice categories is essentially surjective. 

Given a morphism $(r,\phi,\eta):\mathrm{Tr}(H')\rightarrow\mathrm{Tr}(H'')$ in the coslice category of $\mathrm{Tr}(H)$, $r$ is the ratio of the ramification indices of $\mathrm{Tr}(H)\rightarrow\mathrm{Tr}(H'')$ and $\mathrm{Tr}(H)\rightarrow\mathrm{Tr}(H')$.  By the discussion preceding Proposition \ref{hfmorphcons}, this is the ratio of $\frac{\log \theta_H}{\log\theta_{H''}}$ and $\frac{\log \theta_H}{\log\theta_{H'}}$. Hence $r=\frac{\log \theta_{H'}}{\log\theta_{H''}}$.  The functor between the coslice categories is full by Corollary \ref{fullcor}.
\end{proof} 

The statement of Theorem \ref{delhypthm} uses isomorphisms of hyperfields rather than isomorphisms of valued hyperfields.  It is easy to see that an isomorphism of valued hyperfields is simply a hyperfield isomorphism which preserves the absolute value. Hence, we would like the following result, which says that an isomorphism of hyperfields between two discretely valued hyperfields which aren't fields almost preserves this absolute value.

\begin{prop}\label{valirrel}Let $H,H'$ be discretely valued hyperfields which are not fields, and let $f:H\rightarrow H'$ be an isomorphism of hyperfields.  Then $|x|_H'=|x|_{H}^{\frac{\log{ \theta_{H'}}}{\log{\theta_{H}}}}$ for all $x\in H$. Furthermore, $\mathrm{Tr}(H)\cong \mathrm{Tr}(H')$. 
\end{prop}
\begin{proof}Before proceeding with the proof, it is best to establish the following claim:  If $H$ is a discretely valued hyperfield which is not a field, then $|x|\leq |y|$ if and only if $(x-x)\subseteq (y-y)$.  The key to proving this claim is to note that $x-x$ is a closed ball of radius $\rho_H|x|$ centered at $0$.  Let $z\in H$ be an element with $|z|=\rho_H$; this is possible because $\rho_H$ is a power of $\theta_H$.  Suppose $x-x\subseteq y-y$. Then the closed ball of radius $\rho_H|x|$ centered at $0$ is contained in the closed ball of radius $\rho_H|y|$ so $xz$ is in this ball. Thus $\rho_H|x|=|xz|\leq \rho_H|y|$ and so $|x|\leq |y|$.  The reverse direction of the claim is similar.

Let $x,y\in H$.  Since $f$ is an isomorphism, $x-x\subseteq y-y$ if and only if $f(x)-f(x)\subseteq f(y)-f(y)$.  Hence $|x|\leq |y|$ if and only if $|f(x)|\leq |f(y)|$.  In particular, $|x|=|y|$ if and only if $|f(x)|=|f(y)|$.  Let $\pi\in H$ be a uniformizer.  For any nonzero $x\in H$, we have $|x|=\theta_H^{v(x)}=|\pi^{v(x)}|$.  Hence $|f(x)|=|f(\pi)|^{v(x)}=|x|^{\frac{\log{|f(\pi)|}}{\log{\theta_H}}}$.  To establish the first part of the proposition we must show that $|f(\pi)|=\theta_{H'}$. For all $x\in H$ we either have $|x|\leq |\pi|$ or $|x|\geq |1|$ so that either $|f(x)|\leq |f(\pi)|$ or $|f(x)|\geq |1|$.  Hence for all $y\in H'$ we have $|y|\leq |f(\pi)|$ or $|y|\geq |1|$.  Hence $f(\pi)=\theta_{H'}$ and the first claim of the proposition is established. 

For the  second part, we apply the same trick used in Remark \ref{rescaleabs}.  Let $\mathring{H}'=H'$ as hyperfields, but equip it with the absolute value rescaled in such a way that $\theta_{\mathring{H}'}=\theta_H$.  This does not affect the associated triple so $\mathrm{Tr}(\mathring{H}')=\mathrm{Tr}(H')$.  Now the isomorphism of hyperfields $f:H\rightarrow\mathring{H}'$ preserves absolute value so $\mathrm{Tr}(\mathring{H}')\cong\mathrm{Tr}(H)$.  Hence $\mathrm{Tr}(H)\cong\mathrm{Tr}(H')$.
\end{proof}

\begin{cor}Let $K$ be a local field.  Then the category of finite separable extensions $L/K$ satisfying $\Gal(K^s/L)\supseteq Gal(K^s/K)^i$ depends only on the isomorphism class of the hyperfield $K/(1+\mathfrak{m}_K^i)$.\end{cor}
\begin{proof}P. Deligne has shown in \cite{deligne} that the category of finite separable extensions $L/K$ satisfying $\Gal(K^s/L)\supseteq Gal(K^s/K)^i$ depends only upon the triple $(\mathcal{O}_K/\mathfrak{m}_K^i,\mathfrak{m}_K/\mathfrak{m}_K^{i+1},\epsilon)=\mathrm{Tr}(K/(1+\mathfrak{m}_K^i))$ associated to $K$.  By Proposition \ref{valirrel}, this in turn depends only on the isomorphism class of the hyperfield $K/(1+\mathfrak{m}_K^i)$.
\end{proof}

\section{Finite flat morphisms}
P. Deligne has given definitions of finite and flat morphisms of triples.  Finite flat morphisms of triples play in important role in \cite{deligne}, where it is shown that the category of finite separable extensions $L$ of a local field $K$ satisfying $\Gal(K^s/K)^i\subseteq \Gal(K^s/L)$ is equivalent to a category whose objects are finite flat extensions  of $\mathrm{Tr}_i(K)$ satisfying a certain ramification condition and whose morphisms are equivalence classes of morphisms of triples.  Our goal in this section is to describe what finiteness and flatness of morphisms mean in the context of discretely valued hyperfields.

\begin{defn}\cite{deligne}Given a ring $R$, we let $l(R)$ denote its length as a module over itself.  Let $(r,\phi,\eta):(R,M,\epsilon)\rightarrow (R',M',\epsilon')$ be a morphism of triples.  The morphism is called flat if $l(R')=l(R)r$. It is finite if $\phi:R\rightarrow R'$ is.\end{defn}

The analogous definitions for valued hyperfields are the following.

\begin{defn}\label{finflatdef}We say a morphism $f:H\rightarrow H'$ is finite if there is a finite subset $S\subseteq \mathcal{O}_{H'}$ such that for all $x\in \mathcal{O}_{H'}$, there is a family of elements $a_s\in \mathcal{O}_H$ indexed by $S$ such that $x\in\displaystyle\sum_{s\in S} f(a_s)s$.  We say that a morphism $f$ is flat if $\rho_H=\rho_{H'}$.\end{defn}

We now check the definitions for discretely valued hyperfields and for triples coincide.

\begin{prop}\label{finflat1}Let $f:H\rightarrow H'$ be a morphism of discretely valued hyperfields which are not fields.  Then the morphism of triples $\mathrm{Tr}(f)$ is finite if $f$ is finite and is flat if $f$ is flat.  In particular, $\mathrm{Tr}$ induces a functor from the category of discretely valued hyperfields which are not fields with finite flat morphisms to the category of triples and finite flat morphisms.\end{prop}
\begin{proof}
Suppose $f$ is finite.  Then let $S\subseteq \mathcal{O}_{H'}$ be a finite generating set, i.e. a finite subset such that for all $x\in \mathcal{O}_{H'}$, there is a family of elements $a_s\in\mathcal{O}_{H}$ indexed by $S$ such that $x\in\displaystyle\sum_{s\in S} f(a_s)s$.    Let $R$ and $R'$ be the truncated DVRs associated to $H$ and $H'$, as in Lemma \ref{Rtdvr}.  Let $\bar{x}\in R'$.  Let $x\in\mathcal{O}_{H'}$ be a lift.  Then we can pick elements $a_s\in \mathcal{O}_H$ such that $x\in \displaystyle\sum_{s\in S} f(a_s)s$.  Reducing to the quotient and defining $\phi$ in the same manner used just above Proposition \ref{phiwelldef} gives $\bar{x}\in \displaystyle\sum_{s\in S} \phi(\overline{a_s})\bar{s}$.  Hence $\phi$ makes $R'$ into an $R$-module which is generated by the finite set $\{\bar{s}\mid s\in S\}$.

Suppose $f$ is flat.  Then $\rho_H=\rho_{H'}$ and we denote the common value by $\rho$. By Lemma \ref{Rtdvr}, $\frac{\log \rho}{\log\theta_H}=l(R)$ and $\frac{\log \rho}{\log\theta_{H'}}=l(R')$ so $l(R')=\frac{\log\theta_H}{\log\theta_{H'}}l(R)=rl(R)$.  Hence $\mathrm{Tr}(f)$ is flat.
\end{proof}

\begin{thm}$f$ is finite and flat if and only if $\mathrm{Tr}(f)$ is.\end{thm}

\begin{proof}We only need to show that finite flat morphisms of discretely valued hyperfields correspond to finite flat morphisms of triples.  For flatness, one uses the flatness part of  Propostion \ref{finflat1} and its converse (which is proven in the same way).  We have seen that if $f$ is finite then so is $Tr(f)$.  Let $f:H\rightarrow H'$ be a morphism of triples such that $\mathrm{Tr}(f)$ is finite.  Let $\mathrm{Tr}(H)=(R,M,\epsilon)$ and $\mathrm{Tr}(H')=(R',M',\epsilon')$, so $\phi:R\rightarrow R'$ is finite; the generators will be denoted $\bar{\alpha_1},\ldots,\bar{\alpha_n}$, and their lifts in $\mathcal{O}_{H'}$ will be denoted $\alpha_1,\ldots,\alpha_n$.  
Let $x\in H'$ have absolute value 1.  Then there are $a_1,\ldots,a_n\in \mathcal{O}_H$  such that $\bar{x}=\phi(\bar{a_1})\bar{\alpha_1}+\ldots+\phi(\bar{a_n})\bar{\alpha_n}$.  Using the definition of addition in $R'$, there exists $y\in H'$ such that $\bar{x}=\bar{y}$ and $y\in f(a_1)\alpha_1+\ldots f(a_n)\alpha_n$.  Because $|x|=1$, $x$ is the unique lift of $\bar{x}$, and so $x\in f(a_1)\alpha_1+\ldots+f(a_n)\alpha_n$.

We now let $x\in \mathcal{O}_H'$ be arbitrary.  Let $r$ be the ramification index.  Let $\pi_H\in H$ and $\pi_{H'}\in H'$ be uniformizers.  Then there exist $i,j>0$ such that $|f(\pi_H)^{i}\pi_{H'}^{j}|=|x|$ and $0\leq j<r$.  Then applying the above argument to $f(\pi_H)^{-i}\pi_{H'}^{-j}x$ shows there are $a_1,\ldots,a_n$ such that $x\in f(\pi_H^{i}a_1)\alpha_1\pi_{H'}^j+\ldots+f(\pi_H^{i}a_n)\alpha_n\pi_{H'}^j$.  Note that for all $k$, $\pi_h^i a_k\in \mathcal{O}_H$ and $\pi_{H'}^j\alpha_k\in \mathcal{O}_{H'}$, and so $x$ is in the span of the family of elements $\alpha_k\pi_{H'}^j$ with $0\leq j<r$.  Hence $f$ is finite.\end{proof}


\begin{thebibliography}{1}
\bibitem{monoidstohyper}A. Connes, C. Consani, \textit{From monoids to hyperstructures: in search of an absolute arithmetic}, in Casimir Force, Casimir Operators and the Riemann Hypothesis, de Gruyter (2010), pp. 147-198.
\bibitem{thickreals}A. Connes, C. Consani, \textit{The universal thickening of the field of real numbers}, arXiv:0420079v1 [mathNT] (2012).  To appear
on Advances in the Theory of Numbers Thirteenth Conference of the Canadian Number
Theory Association (2014).
\bibitem{adeleclasses}A. Connes, C. Consani, \textit{The hyperring of ad\`ele classes},  Journal of 
Number Theory 131.2 (2011): 159-194.
\bibitem{deligne}P. Deligne, \emph{Les corps locaux de caract\'eristique p, limits de corps locaux de caract\'eristique 0}, J.-N. Bernstein, P. Deligne, D. Kazhdan, M.-F. Vigneras, Representations des groupes reductifs sur un corps local, Travaux en cours, Hermann, Paris, pp.119-157,
1984).
\bibitem{krasner}Marc Krasner, \emph{Approximation des corps valu\'es complets de caract\'eristique $p\neq 0$ par ceux caract\'eristique 0}, 1957 Colloque d'alg\'ebre sup\'erieure, tenu \'a Bruxelles du 19 au 22 d\'ecembre.
\bibitem{lyndon}R. C. Lyndon, \textit{Relation algebras and projective geometries}, Michigan Math. J. 8 1961 21-28.
\bibitem{marshall}M. Marshall, \textit{Real reduced multirings and multifields}, Journal of Pure and Applied Algebra, Volume 205, Issue 2, pp. 452-468, 2006.
\bibitem{marshallreview}M. Marshall, \textit{Book Review: Valuations, orderings and Milnor $K$-theory}, Bulletin of the American Mathematical Society (Impact Factor: 2.11). 07/2007; 45(3):439-444. DOI: 10.1090/S0273-0979-07-01166-4.  Available at \url{http://www.ams.org/journals/bull/2008-45-03/S0273-0979-07-01166-4/S0273-0979-07-01166-4.pdf}.
\bibitem{marty} F. Marty, \textit{Sur une g\'en\'eralisation de la notion de groupe}, Huiti\`eme Congres des Math\'ematiciens, Stockholm 1934, 45--59.
\bibitem{prenowitz} W. Prenowitz, \textit{Projective Geometries as Multigroups}, American Journal of Mathematics, Vol. 65, No. 2 (1943), pp. 235--256.
\bibitem{serre}Jean-Pierre Serre, 
\emph{Local Fields}.
Springer-Verlag, 
2nd edition, 1995.
\bibitem{thesis}J. Tolliver, \textit{Hyperstructures and Idempotent Semistructures}, PhD thesis, Johns Hopkins University, 2015.
\bibitem{viro}O. Viro, \textit{Hyperfields for tropical geometry I, hyperfields and dequantization}, arXiv:1006.3034v2.
\bibitem{wintenberger}
J.-P. Wintenberger, \emph{Le corps des normes de certaines extensions infinies des corps locaux; applications}, Annales Scientifiques de l'\'{E}cole Normale Sup\'{e}rieure, S\'{e}r. 4, 16 no. 1 (1983), pp. 59-89.
\end{thebibliography}
\end{document}